\documentclass[amsthm]{elsart}

\usepackage{yjsco}

\usepackage{enumitem} 
\usepackage{xspace} 

\usepackage{amsmath,amssymb} 
\usepackage{amsthm}
\usepackage{bm}   
\usepackage{pdfsync}  
\usepackage{subfigure}
\usepackage{color}
\usepackage[english]{babel}

\usepackage[T1]{fontenc}
\usepackage{yfonts}

\renewcommand{\labelenumi}{(\alph{enumi})}  

\theoremstyle{plain}
\newtheorem{theorem}{Theorem}[section]
\newtheorem{proposition}[theorem]{Proposition}
\newtheorem{lemma}[theorem]{Lemma}
\newtheorem{corollary}[theorem]{Corollary}

\theoremstyle{definition}
\newtheorem{example}[theorem]{Example}
\newtheorem{definition}[theorem]{Definition}
\newtheorem{remark}[theorem]{Remark}

\let\phi=\varphi
\let\rho=\varrho
\let\theta=\vartheta
\let\epsilon=\varepsilon

\def\LT{\mathop{\rm LT}\nolimits}

\def\Supp{\mathop{\rm Supp}\nolimits}
\def\numer{\mathop{\rm numer}\nolimits}
\def\denom{\mathop{\rm denom}\nolimits}

\def\den{\mathop{\rm den}\nolimits}

\def\y1{\mathbf{y}^{(1)}}

\newcommand \ie {\textit{i.e.}}
\newcommand \eg {\textit{e.g.}}
\newcommand \ideal[1] {\langle #1 \rangle}
\newcommand \implicit {\mathop{\rm Implicit}}
\newcommand \monic {\mathop{\rm monic}}

\newcommand \QQ {{\mathbb Q}}

\newcommand \ZZ {{\mathbb Z}}
\newcommand \LPP {\mathop{\rm LPP}\nolimits}
\newcommand \PP {{T}}

\newcommand{\Kt}{{K[t_1,\dots,t_s]}}
\newcommand{\Kx}{{K[x_1,\dots,x_n]}}

\newcommand{\PPList}{\textit{PPL}}
\newcommand{\PhiQB}{\textit{PhiQB}}
\newcommand{\QB}{\textit{QB}}

\newcommand{\compatto}{{enumerative}\xspace}

\def\To{\longrightarrow}
\def\TTo#1{\mathop{\longrightarrow}\limits^{#1}}

\let\To=\longrightarrow
\def\TTo#1{\mathop{\longrightarrow}\limits ^{#1}}

\def\calc{_{\rm calc}}
\def\crt{_{\rm crt}}
\def\hom{^{\rm hom}}
\def\homh{^{{\rm hom}(h)}}
\def\homhD#1{^{{\rm hom}_{#1}(h)}}
\def\homx0{^{{\rm hom}(x_0)}}
\def\deh{{}^{\rm deh}}
\def\dehh{{}^{{\rm deh}(h)}}
\def\dehx0{{}^{{\rm deh}(x_0)}}

\def\longiso{\,\smash{\TTo{\lower 7pt\hbox{$\scriptstyle\sim$}}}\,}

\def\tfrac #1#2{{\textstyle\frac{#1}{#2}}}

\def\cocoa{\mbox{\rm
C\kern-.13em o\kern-.07 em C\kern-.13em o\kern-.15em A}}
\def\apcocoa{\mbox{\rm
A\kern-0.13em p\kern -0.07em C\kern-.13em o\kern-.07 em C\kern-.13em
o\kern-.15em A}}

\begin{document}

\begin{frontmatter}

\title{Implicitization of Hypersurfaces}

\author{John Abbott}
\address{\scriptsize Dip. di Matematica, 
\ Universit\`a degli Studi di Genova, \ Via
Dodecaneso 35,\
I-16146\ Genova, Italy}
\ead{abbott@dima.unige.it}

\author{Anna Maria Bigatti}
\address{\scriptsize Dip. di Matematica, 
\ Universit\`a degli Studi di Genova, \ Via
Dodecaneso 35,\
I-16146\ Genova, Italy}
\ead{bigatti@dima.unige.it}

\author{Lorenzo Robbiano}
\address{\scriptsize Dip. di Matematica, 
\ Universit\`a degli Studi di Genova, \ Via
Dodecaneso 35,\
I-16146\ Genova, Italy}
\ead{robbiano@dima.unige.it}

\thanks{This research was partly supported by the ``National Group for
Algebraic and Geometric Structures, and their Applications''(GNSAGA --
INDAM)}

\begin{abstract}
  We present new, practical algorithms for the hypersurface
  implicitization problem: namely, given a parametric description (in
  terms of polynomials or rational functions) of the hypersurface,
  find its implicit equation.  Two of them are for polynomial
  parametrizations: one algorithm, ``ElimTH'', has as main step the
  computation of an elimination ideal via a \textit{truncated,
    homogeneous} Gr\"obner basis.  The other algorithm, ``Direct'',
  computes the implicitization directly using an approach inspired by
  the generalized Buchberger-M\"oller algorithm.  Either may be used
  inside the third algorithm, ``RatPar'', to deal with
  parametrizations by rational functions.  Finally we show how these
  algorithms can be used in a modular approach, algorithm
  ``ModImplicit'', for avoiding the high costs of arithmetic with
  rational numbers.  We exhibit experimental timings to show the
  practical efficiency of our new algorithms.
\end{abstract}

\date{\today}


\end{frontmatter}


\section{Introduction}

Let $K$ be a field and let $P = K[x_1, \dots, x_n]$ be
a polynomial ring in $n$ indeterminates. Then 
let  $f_1, ..., f_n$ be elements in the field  $L=K(t_1, \dots, t_s)$,
where $\{t_1, \dots, t_s\}$ is another set of indeterminates which are 
viewed as parameters.  We consider the $K$-algebra homomorphism
$$
\phi: K[x_1, ..., x_n] \To K(t_1, \dots, t_s) \hbox{\rm \quad given by\quad } 
x_i \mapsto f_i \hbox{\rm \quad  for\ }  i=1, \dots, n
$$ 
Its kernel, which will be denoted by $\implicit(f_1, \dots, f_n)$, is a prime ideal,
and the general problem of implicitization is to 
find a set of generators for this ideal.

The task of computing $\implicit(f_1, \dots, f_n)$
can be solved by computing a suitable Gr\"obner basis 
(see Proposition~\ref{prop:implicitintersection}).
However, in practice this method 
does not work well since in most non-trivial cases it is far too slow.
The poor computational speed is aggravated when computing with 
rational coefficients (rather than coefficients from a finite field).

There is definitely a big need for new, efficient techniques, and many authors have investigated
alternative ways. The literature about implicitization is so vast that it is almost impossible
to mention the entire body of research on this topic. 
This interest derives from the fact that the parametric representation of a rational variety is 
important for generating points on it, while the
implicit representation is used to check whether a point lies on it.  
Besides its theoretical importance, the double representation of a rational variety 
is used intensively for instance in Computer Aided 
Geometric Design. A~good source of bibliography up to ten years ago is~\cite{BJ}.
More recently,  new ideas have emerged. 
As we said, it is almost impossible to cite all of them, and we content ourselves to mention a few. 
In particular, new methods for computing implicitizations have been 
described in~\cite{BB}, \cite{BC},  \cite{CdA}, and \cite{OR}.
Some of these new ideas respond to the fact that in many cases the computation of
$\implicit(f_1, \dots, f_n)$ is too hard, hence one seeks a way to check whether
a point lies on the rational variety without actually computing its equations.

\smallskip
So, what is the content of this paper? And what are the novelties and 
the new algorithms presented here?
First of all, we concentrate on the ``hypersurface case'' where
$\implicit(f_1, \dots, f_n)$ is a
principal ideal, and hence generated by an irreducible polynomial which is 
therefore unique up to an invertible constant factor.  

\begin{remark}\label{rem:s=n-1}
Let~$I$ be the ideal $\implicit(f_1, \dots, f_n)$
and let~$m$ be the
minimum number of generators for~$I$.
From the facts that
$\dim(K[f_1,\dots,f_n]) \le \dim(K[t_1,\dots,t_s]) = s$
and
$\dim(K[f_1,\dots,f_n]) = \dim(K[x_1,\dots,x_n]/I) \ge n-m$ 
it follows that $m \ge n-s$.
So, whenever $s \le n-1$ then~$I$ has at least one generator,
and in particular~$I$ is non-zero.
\end{remark}

The hypersurface case typically arises when $s = n-1$,
in accordance with the remark above.
However, this is not always the case, as the following examples show. 

\begin{example}\label{ex:bad-parametrizations}
We consider the ``atypical'' case where $n=s=2$ and $f_1=f_2=t_1{+}t_2$. 
Clearly we have $\implicit(f_1, f_2)= \ideal{ x-y }$, which is obviously principal.
This does, however, become a typical case if we use
a ``better parametrization'' in terms of $u = t_1{+}t_2$, 
where we then have $f_1 = f_2 =u$, and consequently also have $s = n-1$ with this
better parametrization.

Another ``atypical'' example is the following.
Let $n=3$, $s=2$ and $f_1 = \frac{t_2^2}{t_1^2}$,\;
$f_2 = \frac{t_1^2+t_2^2}{t_1^2}$,\;
$f_3 = \frac{t_1^2+t_1t_2+t_2^2}{t_1^2}$.
Here we do have $s = n-1$ but the implicitization is not principal, in fact
it turns out that 
$\implicit(f_1, f_2,f_3)=\ideal{ x_1+x_2-1,\; x_2^2-2x_2x_3+x_3^2-x_2-1 }$.
The reason here is that there is a ``better parametrization'' in terms of
$u =\frac{t_2}{t_1}$, where we have
$f_1 = u^2$,\; $f_2 = u^2+1$,\; $f_3 = u^2+u+1$;
and with this better parametrization we have $s \neq n-1$.
\end{example}
 
 In this paper we do not examine the interesting question of finding a
good parametrization, which is a problem of a quite different nature.

\smallskip
\noindent
The ideas explored in this paper can be summarized in the following way:
\begin{itemize}[parsep=1pt]
\item
  Exploit homogenization to improve elimination (RatPar, ElimTH):\\
  using elimination is known to be an elegant but impractical way to
  achieve implicitization.  We show that any problem (polynomial or
  rational parametrization) can be homogenized (see
  Proposition~\ref{prop:homogpoly} and
  Theorem~\ref{th:homogfractions}).  Thereafter, the result is given
  by the first polynomial not involving any of the parameters, so the
  computation can be stopped as soon as it is found,
  avoiding the remaining ``useless reductions''.

\item Use a direct algorithm which does not need elimination
  (Direct):\\
  Wang in~\cite{W} described an algorithm based on searching for a
  linear relationship among the images of the power-products.  We
  refine this idea and make it incremental, thus leading to several
  important insights and opimizations (see
  Subsection~\ref{sec:direct}).

\item
  When the coefficient field is $\QQ$ use modular methods (ModImplicit):\\
  computing the solution polynomial modulo several primes, and then 
  obtaining the solution over $\QQ$ by
  Chinese Remaindering is a powerful tool, but needs to be fine-tuned to any
  specific context.  In Section~\ref{sec:RationalReconstruction} we use
  an incremental approach combined with fault-tolerant rational reconstruction
  to resolve the problem of how many primes are needed and to ``tolerate''
  computations with bad primes (some of which cannot
  be detected \textit{a priori}).
\end{itemize}


\begin{remark}
Regarding the first item, one of the referees pointed out that 
we made no reference to the \textit{``projective view of the implicitization 
problem, which is relatively classical''}.  The main reason was to avoid complications 
for the typical computer-algebra people who, 
generally, are much more familiar with algebra than geometry. 
Nevertheless, let us give some hints in this direction to the interested reader.

If $f_1,\dots,  f_n \in K[t_1,\dots,t_s]$ then the map~$\phi$ can be 
seen as the algebraic counterpart of the 
map of the affine schemes $\Phi: \mathbb A_s \To \mathbb A_n$. 
We let $d_i = \deg(f_i)$ for $i=1, \dots , n$, then homogenize the $f_i$ 
with a new indeterminate $h$ such that $\deg(h) =1$,  and set ${\deg(x_i) = \deg(f_i)}$.
Now we consider the projective space $\mathbb{P}_s$ with coordinates $t_1, \dots, t_s, h$
and the weighted projective space  $\mathbb{P}(d_1, \dots, d_n, 1)$ with 
coordinates $x_1, \dots, x_n, h$ (see for instance~\cite{BR} for an introduction 
to the theory of weighted projective spaces). 
The map~$\Phi$ can be viewed as the restriction to $\mathbb{A}_s$
of the \textit{rational map}
$\Psi: \mathbb{P}_s \dashrightarrow \mathbb{P}(d_1, \dots, d_n, 1)$
given by $[t_1:t_2:\cdots:t_s:h] \to [f_1\homh:f_2\homh:\cdots:f_s\homh:h]$.
Observe that~$\Psi$ is a rational map, but not necessarily a map, since it may 
have a non-trivial base locus.
The algebraic explanation  of this fact is exactly the proof 
of Proposition~\ref{prop:homogpoly}.

The situation is more complicated when $f_1,\dots,  f_n \in K(t_1,\dots,t_s)$.
Using a common denominator, we may assume that $f_i = \frac{p_i}{q}$ with
$f_i, q \in K[t_1, \dots, t_s]$ for $i =1, \dots, n$.
If we let~$D_q$ denote the open subscheme $\mathbb{A}_s\setminus \{q=0\}$, 
then~$\phi$ can be  seen as the algebraic counterpart of the 
map of affine schemes $\Phi: D_q\To \mathbb A_n$, and the standard way to 
proceed is to take care of this limitation, as explained in Remark~\ref{rem:generalfrac}.

But there is a different way to interpret $\implicit(f_1, \dots, f_s)$.
We let $d_i = \deg(p_i)$ for $i=1, \dots, n$ and $d_0 = \deg(q)$.
Then we let $d = \max\{ \deg(q), \deg(p_1), \dots, \deg(p_n) \}$, let $Q = q\homhD{d}$,
$P_i=p_i\homhD{d}$ for $i=1, \dots, n$ (see Definition~\ref{def:homD}), so that 
all the polynomials $Q, P_1, \dots, P_s$ are homogeneous of the same degree $d$.
Next we consider the projective space $\mathbb{P}_s$ with coordinates $t_1, \dots, t_s, h$
and the projective space $\mathbb{P}_n$ with coordinates $x_0, x_1, \dots, x_n$.
If we let $\mathbb{A}_n$ be the affine open chart of $\mathbb{P}_n$ defined by $x_0\ne 0$,
the map $\Phi$ can be interpreted as the restriction to $D_q$ of the corresponding rational map
$\Psi: \mathbb{P}_s \dashrightarrow \mathbb{P}_n$ defined by 
$[t_1:t_2:\cdots:t_s:h] \to [Q:P_1:\cdots:P_s]$.  The algebraic explanation of this 
fact is exactly the proof of Theorem~\ref{th:homogfractions}. 

Why not try to use other embeddings into suitable 
projective or weighted projective spaces, as we do 
in the case of polynomial parametrizations? The reason is explained in all the remarks 
and examples following Theorem~\ref{th:homogfractions}.

\end{remark}


\medskip
\noindent
There is a further idea: computing implicitizations with constraints,
in particular using a method of ``slicing'' the variety with suitable
parallel hyperplanes.
This technique was introduced and used in~\cite{R}.
However, it is rarely better than our new methods when 
$\implicit(f_1, \dots, f_n)$
is a hypersurface.
We shall investigate the ``implicitizations with constraints'' for the general case
in a later paper.

\medskip
The algorithms described in this paper are implemented in
CoCoALib~\cite{cocoalib},
and 
are also available in
 \cocoa\,5~\cite{cocoa5}.
With our new methods
most of the examples mentioned in the literature become ``easy'', that
is we can compute the implicitization in less than a second~---~see
Table~$1$ in Section~\ref{Timings}.
Consequently,
 we introduce new, challenging examples,
and the last table shows the performance of our implementation.

\medskip
We thank the referees for their useful comments and suggestions which
helped us to improve this paper.

\section{Notation and Terminology}

Here we introduce the notation and terminology we shall use.

\begin{definition}
\label{def:elimideal}
  Let $K$ be a field, and let $P = \Kx$ be a polynomial ring
  in~$n$ indeterminates.
Let $t_1, \dots, t_s$ be further indeterminates which are viewed as
``parameters''.
Given elements $f_1, ..., f_n$ in $\Kt$, we define
the ideal $J = \ideal{ x_1-f_1, \dots, x_n-f_n }$ of the ring 
$P[ t_1, \dots, t_s]$ to be the
\textbf{eliminating ideal} of the $n$-tuple $(f_1, \dots, f_n)$.

We define {\boldmath$\implicit(f_1, \dots, f_n)$} to be
the kernel of the $K$-algebra homomorphism
${\phi: P \To \Kt}$ which sends
$x_i \mapsto f_i \hbox{\rm \quad  for\ }  i=1, \dots, n$.
\end{definition}

\begin{definition}\label{def:elimrat}
We extend naturally Definition~\ref{def:elimideal} to parametrizations by rational functions.
Let $\frac{p_1}{q},\dots, \frac{p_n}{q}$ be rational functions in the field $L=K(t_1, \dots, t_s)$
with common denominator $q$;
so that we have $q, p_1, \dots, p_n \in \Kt$.

We define {\boldmath$\implicit(\frac{p_1}{q},\dots, \frac{p_n}{q})$} to be
  the kernel of the $K$-algebra homomorphism
  ${\phi: P \To L}$ which sends
  $x_i \mapsto \frac{p_i}{q}$ \quad  for $i=1, \dots, n$.
\end{definition}

\begin{definition}
  An \textbf{\compatto ordering} is a total ordering such that for every
  element there are only finitely many elements smaller than it.
  In particular, an \textbf{\compatto term-ordering} is a term-ordering
  which is also \compatto; consequently, an \compatto term-ordering
  is defined by a matrix with strictly positive entries in the first row.
\end{definition}

\begin{example}
Any degree-compatible term-ordering is \compatto because for any
power-product~$\bar\PP$ all smaller power-products, $\PP<\bar\PP$,
must have $\deg(\PP)\le\deg(\bar\PP)$, and so they are finite in number.
In contrast, the lex-ordering (for 2 or more indeterminates) is not
\compatto because if indeterminate $x_2$ is less than $x_1$ then
all powers $x_2^d$ are smaller than $x_1$.
\end{example}

\begin{definition}\label{def:homD}

In section~\ref{sec:param-by-ratfns} we shall use two different kinds of homogenization:
\begin{itemize}[topsep=3pt, parsep=1pt]
\item traditional \textbf{homogenization}
 and \textbf{dehomogenization}: with respect to~$h$
  we denote them by the superscripts~$\homh$ and~$\dehh$ respectively;
  and with respect to $x_0$, by the superscripts~$\homx0$ and~$\dehx0$.
\item  \textbf{\boldmath $d$-shifted-homogenization}:
  for a non-zero polynomial~$f$ and degree $d\ge\deg(f)$ we write
  $f\homhD{d}$ to mean 
  $h^{d-\deg(f)} f\homh$, which is a homogeneous polynomial of degree~$d$.
  As a special case, since $0\homh = 0$, we have $0\homhD{d} = 0$ for all~$d$.
\end{itemize}
\end{definition}

The following easy properties of the shifted-homogenization
will help the reader understand the proof of Theorem~\ref{th:homogfractions}
\begin{lemma}
\label{lemma:homogfractions}
  Let $P$ be a polynomial ring over the field~$K$, and let $f,g \in P$.
  \begin{itemize}[topsep=3pt, parsep=1pt]
  \item If $d_1 \ge \deg(f)$ and $d_2 \ge \deg(g)$ then
    $f\homhD{d_1} \cdot g\homhD{d_2} = (fg)\homhD{d_1+d_2}$
  \item If $d \ge \deg(f)$ and $d \ge \deg(g)$ then 
    $f\homhD{d} + g\homhD{d} = (f+g)\homhD{d}$
  \end{itemize}
\end{lemma}

\begin{proof}
The proofs are elementary exercises in algebra.
Observe that the special definition $0\homhD{d} = 0$ is indeed compatible with this
lemma, since from the equality
$f{-}f = 0$ we deduce the equalities $0 = f\homhD{d} - f\homhD{d} = 0\homhD{d}$.
\end{proof}

\section{Polynomial Parametrizations}

In this section we consider a parametrization given by 
polynomials $f_1, \dots, f_n$ in the ring $K[t_1, \dots, t_s]$, where 
$\{t_1, \dots, t_s\}$ is a  set of indeterminates which are viewed as
parameters.  We will look at parametrizations by rational functions
in section~\ref{sec:param-by-ratfns}.



\goodbreak
\begin{proposition}
\label{prop:implicitintersection}
In the setting of Definition~\ref{def:elimideal}:
\begin{enumerate}[parsep=0pt]
\item We have $\implicit(f_1, \dots, f_n) = J\cap P$.

\item The ideal $\implicit(f_1, \dots, f_n)$ can be computed
  using an elimination ordering for all the $t_i$.

\item The ideal $\implicit(f_1, \dots, f_n)$ is prime.
\end{enumerate}
\end{proposition}

\begin{proof}
Claims (a) and (b) are standard results (see for instance book~\cite{KR1}, Section 3.4).
Claim~(c) follows from the
isomorphism $K[t_1, \dots, t_s, x_1, \dots, x_n]/J \cong K[t_1, \dots t_s]$,
whence~$J$ is prime, and so $J\cap P$ is prime too.
\end{proof}

\begin{remark}
\label{rem:exclude-constants}
We shall later find it convenient to assume that in the parametrization
no $x_i$ maps to a constant.  This is not a restriction because if, say, $f_n \in K$
then we obtain the simple decomposition: $\implicit(f_1,f_2,\ldots,f_n) = \ideal{ x_n - f_n } + \implicit(f_1,\ldots,f_{n-1})$.  
Thus any indeterminates $x_i$ which map to constants can simply
be taken out of consideration, letting us concentrate on the interesting part.
Henceforth we shall assume that none of the $f_i$ is constant.
\end{remark}

%
The very construction of the eliminating ideal (in Definition~\ref{def:elimideal})
looks intrinsically non-homogeneous.  And it is well-known that the
behaviour of Buchberger's algorithm 
can be quite erratic when the input is not homogeneous:
usually the computation for a non-homogeneous input is a lot slower than
a ``similar''
homogeneous computation (though there are sporadic exceptions);
for instance, see Example~\ref{ex0}.
We now look quickly at how to use homogenization during implicitization.

A first idea is to give weights to the $x_i$ indeterminates by setting
$\deg(x_i) = \deg(f_i)$ for each~$i$.
If we do so, and if the original $f_i$ are homogeneous polynomials, 
then the eliminating ideal~$J$ turns out to be a homogeneous ideal.
Even when the $f_i$ are not all homogeneous, in the process of ordering and choosing the
power-products of a given degree, we may reasonably expect that
Buchberger's algorithm will ``behave''
more similarly to a homogeneous ideal than with the \textit{standard grading},
where all indeterminates have degree~1.

Although this trick improves the computation in most cases, it is not a
miraculous panacea.  Much better ideas come from the following 
Proposition~\ref{prop:homogpoly} and Theorem~\ref{th:homogfractions}
which reduce the computation of the implicitization ideal to the 
case of prime ideals whose generators are homogeneous polynomials.

\smallskip
In the proofs we use the fundamental properties of homogenization 
and dehomogenization as described in~\cite{KR2}, Section 4.3.
A general discussion about the topic treated in the following proposition
can be found in~\cite{KR2}, Tutorial 51.

\smallskip
In the proposition below we use a single
homogenizing indeterminate~$h$; so, to simplify notation, homogenization and dehomogenization
are tacitly taken with respect to~$h$.

\begin{proposition}\label{prop:homogpoly}{\bf (Implicitizating
Polynomial Parametrizations by Homogenization)}\\
Let $P=K[x_1,\dots,x_n]$, 
let $f_1, \dots, f_n \in \Kt \setminus K$.
Now let~$h$ be a new indeterminate, and let 
$P[t_1, \dots,t_s, h]$ be graded by setting 
$\deg(x_i)=\deg(f_i)$ for $i=1,\dots, n$ 
and $\deg(t_1) =\cdots =\deg(t_s) = \deg(h) = 1$.
Finally let $F_i =f_i\hom$, and
let $\bar{J}$ be the eliminating ideal of $(F_1, \dots, F_n)$.
Then:
\begin{enumerate}
\setlength{\itemsep}{0pt}
\item The ideal $\bar{J}\cap P[h]$ is prime.

\item We have the equality \, $\implicit(f_1, \dots, f_n)= (\bar{J}\cap P[h])\deh$.
\end{enumerate}
\end{proposition}

\begin{proof} 
The proof of claim (a) follows immediately from the fact that $\bar{J}$ 
is an eliminating ideal, hence prime.

Let~$J$ be the (non-homogeneous) eliminating ideal of the tuple $(f_1, \dots, f_n)$. 
Now, since $\bar{J}$ is prime, it is saturated with respect to~$h$;
furthermore we have $\bar{J}\deh=J$, so we can deduce that $\bar{J} = J\hom$.
Clearly $(J\cap P)\hom \subseteq J\hom\cap P[h] = \bar{J}\cap P[h]$, hence, by dehomogenizing, we deduce that 
$J\cap P = ((J\cap P)\hom)\deh \subseteq (\bar{J}\cap P[h])\deh$.
On the other hand, if $f \in  (\bar{J}\cap P[h])\deh$ then we have $f \in \bar{J}\deh\cap P$,
but $\bar{J}\deh = (J\hom)\deh = J$, and the proof is complete 
since $\implicit(f_1, \dots, f_n) = J\cap P$ by Proposition~\ref{prop:implicitintersection}.a.
\end{proof}

%
%

\section{Hypersurfaces Parametrized by Polynomials}
\label{sec:param-by-polys}

In this section we start to treat the ``hypersurface case'', namely
the case where it is known that the implicitization ideal is principal.
In this situation we typically have $s = n-1$, although this equality
is not equivalent to the implicitization being a principal ideal, as
shown in Example~\ref{ex:bad-parametrizations}.

There is no easy way to determine whether the implicitization is going
to be a principal ideal, but this information might already be independently known for the
particular example under consideration.  So this is usually taken as
hypothesis by the papers on this topic.


\subsection{A Truncated Homogeneous Computation}
\label{TruncHomogComp}

As already observed, the ideal $\bar{J}$ in 
Proposition~\ref{prop:homogpoly}  
is homogeneous, hence the computation of the (elimination) Gr\"obner basis  
of~$\bar{J}$ can be performed degree by degree.  Moreover, using the methods
described in Proposition~\ref{prop:homogpoly}
we get the following extra bonus in the hypersurface case:
as soon as we obtain a Gr\"obner basis element, $G$, which does not involve the 
parameters, we may stop the computation of the Gr\"obner basis because
the solution polynomial is just the 
dehomogenization of~$G$.


\begin{corollary}\label{cor:princpoly}
With the same assumptions as in Proposition~\ref{prop:homogpoly},
if \,$\implicit(f_1,\dots,f_n)$ is a principal ideal generated by~$g$
then $\implicit(f_1\hom,\dots,f_n\hom)$ is a principal ideal generated
by~$g\hom$.

\end{corollary}

\begin{proof}
We recall
the equality $\implicit(f_1, \dots, f_n) = (\bar{J}\cap P[h])\deh$ proved in 
Proposition~\ref{prop:homogpoly}.b.  This implies 
that $(\bar{J}\cap P[h])\deh =\ideal{ g }$.  Conversely, 
the ideal $\bar{J}\cap P[h]$ is prime by Proposition~\ref{prop:homogpoly}.a,
hence it is saturated with respect to~$h$, and so it is generated by $g\hom$.
%
\end{proof}


\goodbreak
\begin{algorithm}
\label{alg:ElimTH}
\textbf{(ElimTH: Truncated Homogeneous Elimination)}
\begin{description}
\item[Input] $f_1,\dots,f_n \in \Kt \setminus K$ such that the ideal 
$\implicit(f_1,\dots,f_n)$ is principal.
\item[ElimTH-1] Initialization:
  \begin{description}[topsep=-2pt,parsep=-1pt]
  \item[ElimTH-1.1]
    Create the polynomial ring 
    $R = K[t_1,\dots,t_s,\; h,\; x_1,\dots,x_n]$ graded
    by $[1,\dots,1,\; 1, \; \deg(f_1), \dots, \deg(f_n)]$.\\
    Let $\sigma$ be an elimination ordering for $\{t_1, \dots,t_s\}$ on~$R$.
  \item[ElimTH-1.2]
    Let $F_i = f_i\homh \in R$.
  \item[ElimTH-1.3]
    Let $J = \ideal{ x_1-F_1,\dots,x_n-F_n }$, the eliminating
    ideal of $(F_1,\dots,F_n)$.
  \end{description}
\item[ElimTH-2] \textit{Main Loop:}\\
  Start Buchberger's algorithm for the computation of a $\sigma$-Gr\"obner
  basis~$GB$ of~$J$.\\
  Perform its main loop degree by degree
  (\ie~always choose the lowest degree pair).\\
  When you add to $GB$ the first polynomial~$G$
  such that $\LT_\sigma(G)$ is not divisible by any $t_i$
  exit from loop.
  
\item[ElimTH-3]  Let $g\calc = G\dehh$ mapped into $K[x_1,\dots,x_n]$.

\item[Output] $g\calc \in K[x_1,\dots,x_n]$
\end{description}
Then $g\calc$ generates $\implicit(f_1,\dots,f_n)$.
\end{algorithm}

\begin{proof} 
\textit{Termination}:
The \textit{Main Loop} in the algorithm is simply Buchberger's algorithm, and
that terminates in a finite number of steps.
Moreover, Corollary~\ref{cor:princpoly} guarantees that $\bar{J}$
contains a polynomial not involving the $t_i$ indeterminates,
and since $\sigma$ is an elimination ordering for the $t_i$ there is
such a polynomial in the Gr\"obner basis,
so the \textit{Main Loop} will set~$G$ and exit.

\smallskip
\noindent
\textit{Correctness:}
In the \textit{Main Loop} we execute Buchberger's Algorithm
with respect to an elimination ordering for all the $t_i$;
thus the elements of the Gr\"obner basis whose leading terms
are not divisible by any $t_i$ form a Gr\"obner basis for the
elimination ideal $\bar{J} \cap P[h]$.  By Corollary~\ref{cor:princpoly}
this ideal is principal, so Buchberger's Algorithm (computing degree by degree)
will produce exactly one polynomial whose leading term is not divisible
by any $t_i$.  The \textit{Main Loop} stops as soon as this polynomial is found.
In step \textbf{ElimTH-3} the polynomial~$G$ will be the
generator of $\implicit(f_1\hom,\dots,f_n\hom)$, and by Corollary~\ref{cor:princpoly} we have that $g\calc = G\dehh$ is the generator of $\implicit(f_1,\dots,f_n)$.
\end{proof}

\begin{remark}
  \label{rem:ElimTH-bad-input}

We consider briefly what happens if the input to
Algorithm~\ref{alg:ElimTH}~(ElimTH) does not correspond to a principal
implicitization ideal.  If $\implicit(f_1,\dots,f_n)$ is the zero
ideal then Buchberger's Algorithm in step \textbf{ElimTH-2} will
terminate without finding any candidate for~$G$; we could in that case
simply set $G=0$.  By Remark~\ref{rem:s=n-1} this cannot happen if
$s \le n-1$.

If, instead, $\implicit(f_1,\dots,f_n)$ is non-zero and non-principal
then the polynomial~$G$ found in step \textbf{ElimTH-2} will be a
lowest weighted-degree element of a Gr\"obner basis for that ideal (and
consequently a lowest weighted-degree non-zero element of the ideal).

\end{remark}

\goodbreak
This next example illustrates the good behaviour of the algorithm above. 

\begin{example}\label{ex0}
We let ${K = \mathbb Z/(32003)}$ and
in the ring $K[x_1, x_2, t]$ we consider the 
eliminating ideal 
$$
I = \ideal{ x_1-(t^{15}-3t^2-t+1), \; x_2-(t^{23}+t^{11}+t^3-t-2) }
$$
The usual elimination of~$t$ takes more than one hour, even
if we give the weights $15$ and $23$ to the indeterminates $x_1$, $x_2$ 
respectively; whereas
the truncated homogeneous elimination takes less than a second
(this is one of our test cases: see Example~\ref{ex12ABR-Poly}). 
The solution polynomial has 176 power-products in its support.
\end{example}

\begin{remark}
  If the eliminating ideal is not homogeneous, the idea of truncating
  the computation as soon as a polynomial in~$P$ is found does not
  work well, since it may happen that the first such polynomial computed
  by the algorithm is a proper multiple of the solution
  polynomial.  The phenomenon is similar to the case where the reduced
  Gr\"obner basis of an ideal is~$\{1\}$, yet before discovering that~$1$ is
  in the basis it often happens
  that many other (non-reduced) Gr\"obner basis elements are computed.

  One could take the polynomial found and factorize it, then substitute into
  the various irreducible factors to see which factor is the good one.  But this
  is unlikely to be efficient.
\end{remark}

\subsection{A Direct Approach}
\label{sec:direct}







We briefly recall the setting of this paper:
we have been given a $K$-algebra homomorphism
$\phi{:}\; K[x_1,\ldots,x_n] {\longrightarrow} \Kt$
sending $x_i \mapsto f_i$
and we assume that its kernel is a principal ideal:
the problem is to find the generator of
$\ker(\phi) = \implicit(f_1,\dots,f_n) = \ideal{g}$.
Following Remark~\ref{rem:exclude-constants}, we shall
find it convenient to assume that each $f_i \in \Kt$ is
non-constant.
In this section we compute the polynomial~$g$ via a \textit{direct} approach.

We use the notation {\boldmath $\LPP(f)$} to indicate the leading
power-product of the polynomial~$f$ (also denoted in the literature by
${\rm LT}(f)$ of ${\rm in}(f)$).  If $f=\sum_i a_i\PP_i$, with
distinct power-products $\PP_i$, then the \textbf{support} of~$f$ is
{\boldmath $\Supp(f)$} $= \{ \PP_i \mid a_1\ne 0 \}$.

\begin{remark}
\label{rem:lindep}
First of all notice that,
if a polynomial $f = \sum_i a_i T_i \in \Kx$ is such that $\phi(f)=0$, then
  $\sum_i a_i \, \phi(T_i) = 0$.  
In other words, there is a $K$-linear dependency among
the image polynomials $\{ \phi(\PP) \mid \PP \in \Supp(f)\} \subset \Kt$,
and the coefficients of the
linear relation are exactly the coefficients of~$f$ (up to a scalar
multiple). 
\end{remark}

The idea behind our direct approach is to \textit{directly}
determine $g$ by searching for a linear dependency among all the
$\phi(\PP)$: we generate, one by one, the polynomials $\phi(\PP)$ as $\PP$
runs through the power-products in $\Kx$ until a dependency
exists.  We shall now see how to reduce this apparently infinite
problem to a finite, tractable one.




\goodbreak
\begin{algorithm}
\label{alg:direct}
\textbf{(Direct: Implicitization by Direct Search)}
\begin{description}
\item[Input] $f_1,\ldots,f_n \in \Kt$ such that the ideal 
$\implicit(f_1,\dots,f_n)$ is principal.

\item[Variables]
  {\small
The main variables are:
\begin{description}[topsep=-2pt,parsep=-1pt]
\item[$\QB$:] the list of power-products in $\Kx$ already considered.
\item[$\PPList$:] the list of power-products in $\Kx$ yet to be considered.
\item[$\PhiQB$:] the list $[\phi(\PP_i)\mid\PP_i\in \QB]$; its elements
  are seen as ``sparse vectors'' in the infinite dimensional
  $K$-vector space $\Kt$ with basis comprising all power-products.
\end{description}
  }
  
\item[Direct-1] \textit{Initialization:}
  \begin{description}[topsep=-1pt,parsep=-1pt]
  \item[Direct-1.1]
    Fix an \compatto term-ordering $\sigma$ on $\Kx$.
  \item[Direct-1.2]
    Set $\QB = \emptyset$. \;  Set $\PhiQB = \emptyset$. \;  Set $\PPList = \{1\}$.
  \end{description}

\item[Direct-2] \textit{Main Loop:}
  \begin{description}[topsep=-1pt,parsep=-1pt]
  \item[Direct-2.1]
    Let $\PP = \min_\sigma(\PPList)$. \;
    Remove $\PP$ from $\PPList$.
  \item[Direct-2.2]
    Compute $v = \phi(\PP) \in \Kt$.
    
  \item[Direct-2.3]
    Is there a linear dependency $v = \sum_i a_i  v_i$ with $a_i \in K$ and $v_i \in \PhiQB$?
    \begin{description}[parsep=-3pt]
    \item[yes] exit from loop
    \item[no] \;
    Add to $\PPList$ the elements of
    $\{x_1 \PP,\ldots,x_n \PP\}$ not already in~$\PPList$;\\
      append $\PP$ to the list $\QB$; \;
      append $v$ to the list $\PhiQB$
    \end{description}
  \end{description}

\item[Direct-3]
  Let $g\calc = T - \sum_i a_i \PP_i$
 \quad where $\PP_i\in \QB$ corresponds to $v_i \in \PhiQB$.

\item[Output] $g\calc \in \Kx$
\end{description}
Then $g\calc$ generates $\implicit(f_1,\dots,f_n)$.
\end{algorithm}

\begin{proof} \  
\textit{Termination:}
The main loop of the algorithm considers the power-products
in the ring $\Kx$ in increasing $\sigma$-order until the condition in
step \textbf{Direct-2.3} breaks out; since~$\sigma$ is \compatto, every
power-product will be considered at some (finite) time.  The initial
values for $\QB$ and
$\PPList$, and the updates to these two variables in step \textbf{Direct-2.3\,(no)}
guarantee that whenever we enter step \textbf{Direct-2.1} the set~$\PPList$
satisfies $\PPList = \{ x_i \, \PP : 1 \le i \le n \hbox{ and } \PP \in \QB \}
\setminus \QB$; in other words it comprises those power-products
outside $\QB$ and which border on $\QB$.  As $\sigma$ is a term-ordering, $\PPList$
therefore always contains the $\sigma$-smallest power-product outside $\QB$
(as well as many others).


In step \textbf{Direct-2.3}
the algorithm looks for a $K$-linear dependency amongst
the polynomials $\{\phi(\bar\PP) \mid \bar\PP\le_\sigma\PP\}$.
Every such linear dependency corresponds to a monic element of $\ker(\phi)$.
By hypothesis $\ker(\phi)$ contains the polynomial~$g$
(which we may assume to be monic wrt.~$\sigma$), so if we reach
step \textbf{Direct-2.3} with $\PP = \LPP(g)$ then a linear
dependency will surely be found (\eg~corresponding to the coefficients of~$g$).
Since $\sigma$ is an \compatto ordering,
there are only finitely many power-products less than $\LPP(g)$;
so we will break out of the main loop when $\PP = \LPP(g)$,
if not earlier.

\smallskip
\noindent
\textit{Correctness:}
We shall show that we do not
break out of the main loop until $\PP = \LPP(g)$,
and that when we do break out, the polynomial we construct
in step \textbf{Direct-3} is~$g$.


The test in step~\textbf{Direct-2.3} gives \textit{true} if and only
if there is a polynomial $\tilde{g}$, of the form $\PP - \sum_i a_i \PP_i$ with each
$\PP_i <_\sigma \PP$, satisfying $\phi(\tilde g) = 0$ or equivalently $\tilde g \in \ker(\phi)$.
Note that $\tilde g$ is monic, thus non-zero by construction.

By hypothesis $\ker(\phi)$ is a principal ideal (generated by~$g$).
So every non-zero element
of $\ker(\phi)$ has leading term $\sigma$-greater-than-or-equal to $\LPP(g)$,
thus step \textbf{Direct-2.3} will not find any linear dependency if
$\PP <_\sigma \LPP(g)$.


Let $g\calc$ be the polynomial constructed in step \textbf{Direct-3}.
We have $\LPP(g\calc) = \LPP(g)$, and both polynomials are monic.
Suppose $g\calc \neq g$, and set $\hat{g} = g\calc - g$.  Then
$\LPP(\hat{g}) <_\sigma \LPP(g)$ but also $\phi(\hat{g}) =
\phi(g\calc)-\phi(g) = 0$, so $\hat{g} \in \ker(\phi)$ which
contradicts the fact that~$g$ is the (non-zero) element of $\ker(\phi)$
with $\sigma$-smallest leading term.
This concludes the proof.
\end{proof}

\begin{remark}
  \label{rem:direct-bad-input}
  We consider briefly what happens if the input to
Algorithm~\ref{alg:direct}~(Direct) does not correspond to a principal
implicitization ideal.  

If $\implicit(f_1,\dots,f_n)$ is the zero ideal 
then the \textit{Main Loop} never exits (as no non-trivial linear
dependency exists).
However, if $s \le n-1$ then the ideal $\implicit(f_1,\dots,f_n)$ cannot
be the zero ideal as proved in Remark~\ref{rem:s=n-1}.

If, instead, $\implicit(f_1,\dots,f_n)$ is non-zero and non-principal
then the main loop will exit, and the polynomial~$g\calc$ found in step
\textbf{Direct-3} will be the monic polynomial with $\sigma$-smallest
leading term in the reduced $\sigma$-Gr\"obner basis for that ideal.
\end{remark}

\begin{remark}
  This approach is inspired by the \textit{Generalized
    Buchberger-M\"oller algorithm}~\cite{AKR}, and is somewhat simpler
  (\eg~the list~$G$ for storing the Gr\"obner basis is not needed, and the 
  update to the list~$\PPList$ is
  simpler).  But there is an important difference: here we cannot
  specify \textit{a priori} a finite dimensional vector space as the
  codomain of the \textit{normal form vector map}.  For the
  generalized Buchberger-M\"oller algorithm the finiteness of the
  codomain led to an easy proof of termination; instead here we had to
  introduce the concept of \compatto ordering.
\end{remark}

\begin{remark}[Optimizations]
\label{rem:DirectOpt}
We mention here a few important optimizations which considerably improve the
execution time:
\begin{enumerate}
\item
  The successive linear systems we check in step~\textbf{Direct-2.3} are
  very similar: in practice we build up a row-reduced matrix adjoining a
  new row on each iteration.
\item
  The computation of $\phi(\PP)$ in step~\textbf{Direct-2.2} can be
  effected in several ways.  We suggest exploiting the fact that
  $\phi$ is a homomorphism to compute the value cheaply.  Apart from
  the very first iteration when $\PP = 1$, we always have $\PP = x_j
  \, \PP'$ for some indeterminate $x_j$ and some power-product $\PP'$
  for which we have already computed $\phi(\PP')$; so we can calculate
  with just a single multiplication $\phi(\PP) = \phi(x_j) \,
  \phi(\PP')$.  Usually there are several choices for the
  indeterminate $x_j$, so we can choose the one which leads to the
  cheapest multiplication.  Note that in step~\textbf{Direct-2.3(no)}
  we manipulate just power-products when updating $\PPList$.

\item
In step~\textbf{Direct-1.1} we pick some \compatto ordering on the
power-products of the ring $\Kx$.  Here we describe a specific good
choice; the idea is that as we pick (in step~\textbf{Direct-2.1}) the
power-products $\PP$ in increasing order then the corresponding
$\LPP(\phi(\PP))$ are in non-decreasing order.

We start with a (standard) degree-compatible term-ordering $\tau$ on
the power-products of $\Kt$.  Let $M_\tau$ be an $s\times s$ integer matrix
representing it (so all entries in the first row are~$1$).  We define
the \textbf{order vector} of a power-product $t_1^{e_1} t_2^{e_2} \cdots t_s^{e_s}$
to be $M_\tau \underline{e}$; the ordering $\tau$ is thus equivalent to
lex comparison of the order vectors.


Let $E$ be the $s\times n$ integer matrix whose columns are the
exponents of $\LPP_\tau(f_i)$; put $M = M_\tau \cdot E$, an $s\times
n$ matrix whose $i$-th column is the order vector of
$\LPP_\tau(\phi(x_i))$.  The first row of $M$ is strictly positive:
the $i$-th entry is $\deg(f_i)$.  We complete $M$ to a term-ordering
matrix $M'$ for the power-products of $\Kx$: \ie~we remove rows
linearly dependent on those above it, and adjoin new rows at the
bottom to make~$M'$ square and invertible.  The
term-ordering defined by $M'$ is \compatto since~$M$ and~$M'$ have the
same first row, and it has our desired property.

\end{enumerate}
\end{remark}

\begin{example}
  An example to illustrate Remark~\ref{rem:DirectOpt}(b).  Let $K[s,t]$
  have terms ordered by $\bigl(\begin{smallmatrix}1&1\\1&0\end{smallmatrix}\bigr)$.
    Let $f_1= s^5 -st^3 -t, \; f_2= st^2 -s,\; f_3= s^4 -t^2$ then the
    order vectors of the LPPs are $(5,5)$, $(3,1)$ and $(4,0)$ respectively.
    So we obtain the matrix $M' = \left( \begin{smallmatrix}5&3&4\\5&1&0\\\ast&\ast&\ast\end{smallmatrix}\right)$ where we can fill the last row freely to make the matrix invertible, \eg~$(0\;0\;1)$.
\end{example}

\begin{remark}
  We contrast Algorithm~\ref{alg:direct}~(Direct) with the method presented by Wang in~\cite{W}.
  The underlying idea is the same: find the generator by searching for a linear
  relationship among the images of power-products.  Wang's method adjoins new power-products
  in blocks.  Each block comprises all power-products of a given standard degree
  (where each indeterminate has degree~1).  Wang observed that the linear systems
  produced ``tend to be almost triangular''.

  Our approach adjoins new power-products one at a time, and this lets us
  use several important optimizations (described in
  Remark~\ref{rem:DirectOpt}).  On each iteration we do a single
  multiplication to obtain $\phi(\PP)$, and then a single ``row reduction''.
  Using the term-ordering described in Remark~\ref{rem:DirectOpt}(c)
  guarantees that our linear system is as triangular as possible, and by
  adjoining power-products one by one we keep the system small (and avoid
  computing extraneous 
  images of power-products under $\phi$).

  The importance of these optimizations is illustrated by the computation
  time for Example~\ref{Wang}: our implementation of Algorithm~\ref{alg:direct}~(Direct)
  took less than 8 seconds, while Wang reported about 47000 seconds~---~no
  doubt some (but not all) of the speed gain is due to improvements in hardware.
\end{remark}

\begin{remark}
  The requirement that $\sigma$ be a term-ordering is stronger than
  necessary.  For instance, it is sufficient that $\sigma$ orders by degree
  (how it orders within a fixed degree is unimportant).  However, if we use
  such a general ordering then the list~$\PPList$
  will then have be to updated differently in step \textbf{Direct-2.3(no)}
  (\eg~fill it with all power-products of the next degree when it becomes
  empty, much like Wang's method~\cite{W}).
\end{remark}

\section{Hypersurfaces Parametrized by Rational Functions}
\label{sec:param-by-ratfns}

In this section we consider a parametrization given by 
rational functions $f_1, \dots, f_n$ in the field $L=K(t_1, \dots, t_s)$, where 
$\{t_1, \dots, t_s\}$ is a  set of indeterminates which are viewed as
parameters.  
We can write this parametrization with a common denominator $q$,
so that we have $f_i = \frac{p_i}{q}$ with 
$q, p_1, \dots, p_n \in K[t_1, \dots,t_s]$.


\begin{remark}\label{rem:generalfrac}
We recall here a general 
method for computing $\implicit(\frac{p_1}{q}, \dots, \frac{p_n}{q})$.
We start with the ideal $I=\ideal{ qx_1-p_1, \dots, qx_n-p_n }
\subset K[x_1,\dots,x_n, \; t_1, \dots,t_s]$,
then we introduce a new indeterminate $u$,
and let $J = I + \ideal{ uq-1 }
\subset K[x_1,\dots,x_n, \; t_1, \dots,t_s, \; u]$.
Now we have $\implicit({\frac{p_1}{q}}_{\mathstrut}, \dots, \frac{p_n}{q}) = J\cap P$ which can be computed by eliminating the indeterminates
$u$ and $t_1, \dots, t_s$. The following example illustrates the necessity of
adding $\ideal{ uq-1 } $ to~$I$.
\end{remark}

\begin{example}\label{uqminusone}
We let $K = \mathbb Q$ and
we let $f_1 = f_2 = \frac{s}{t}, \,f_3 = s$ in $K[s,t]$.
We construct the ideal $I = \ideal{tx-s, \  ty-s, \ z-s}$,
and observe that
$I = \ideal{x-y}\cap \ideal{s,t,z}$.  Hence we get
$I \cap K[x,y,z] = \ideal{z\,(x-y)}$
which is not prime.
The correct result, which may be obtained by including the
generator $ut-1$, is $\implicit(f_1, f_2, f_3) =\ideal{ x-y }$.
\end{example}

In the following 
we present a ``homogeneous'' method for the computation of the ideal $\implicit(f_1, \dots, f_n)$
when $f_1,\dots, f_n$ are rational functions; compared to the classical
method the big advantage we have is that the Gr\"obner basis computation
(\ie~the elimination) is performed on a homogeneous ideal.

\begin{theorem}
\label{th:homogfractions}
  {\bf(Implicitizating Rational Parametrizations by Homogenization)}\\
Let $q, p_1, \dots, p_n$ be non-zero polynomials in 
$K[t_1, \dots, t_s]$.
We shall work in the graded ring
 $R = K[x_0,x_1, \dots, x_n, t_1, \dots, t_s, h]$
with grading defined by
${\deg(h) = \deg(t_i) = 1}$ for $i=1, \dots, s$ and
 $\deg(x_i)=d$ for $i = 0,\dots, n$ where $d =\max\{\deg(q), \deg(p_1), \dots, \deg(p_n)\}$.

\smallskip
Now make the input polynomials homogeneous and of equal degree~$d$:
set $Q=q\homhD{d}$, and set
$P_i = p_i\homhD{d}$ for $i =1, \dots, n$.
Let $\bar{J}$ be the homogeneous eliminating ideal in the ring~$R$ generated by   
$\{x_0-Q,\; x_1-P_1, \dots, x_n-P_n\}$.  Then:
\begin{enumerate}[topsep=3pt,parsep=0pt]
\item The ideal $\bar{J}\cap K[x_0,x_1,\dots, x_n]$ is prime.

\item We have the equality $\implicit(\frac{p_1}{q}, \dots, \frac{p_n}{q})
=\big(\bar{J}\cap K[x_0,x_1,\dots, x_n]\big)\dehx0$.
\end{enumerate}
\end{theorem}

\begin{proof} 
The kernel $\implicit(\frac{p_1}{q}, \dots, \frac{p_n}{q})$ is an ideal in the ring
$P = K[x_1, \dots, x_n]$, which we view as a subring of~$R$.
We write $P[x_0]$ to denote the subring $K[x_0, x_1, \ldots, x_n]$:
observe that a polynomial in $P[x_0]$ is homogeneous in
the induced grading if and only if
it is homogeneous in the standard grading (\ie~in the usual sense of the word).

\smallskip
As in Proposition~\ref{prop:homogpoly}, the proof of claim~(a) follows immediately from
the fact that $\bar{J}$ is an eliminating ideal, hence prime. 

\medskip
\noindent
To prove claim (b) we introduce the following sets:
\begin{itemize}[topsep=2pt, parsep=-0pt]
\item $S_1=   \{ A \in P \ \mid \ \ 
  {A(\frac{p_1}{q}, \dots,  \frac{p_n}{q})} = 0\}$


\item $S_2 = \{ A\in P[x_0] \ \mid \ A \text{ is homogeneous and \ } 
{A(Q, P_1, \dots, P_n)}= 0\}$

\item $S_3=   \{ A\in P[x_0]  \ \mid \ A \text{ is homogeneous and \ } 
{A(q, p_1, \dots, p_n)} = 0\}$

\item $S_4=   \{ A\in P[x_0]  \ \mid \ A \text{ is homogeneous and \ } 
{A\dehx0(\frac{p_1}{q}, \dots, \frac{p_n}{q})}= 0\}$
\end{itemize}

\noindent
Clearly, the conclusion is reached if we prove the following claims.
\begin{enumerate}[topsep=2pt, parsep=0pt]
\renewcommand{\labelenumi}{(\arabic{enumi})}
\item The set $S_1$ is the ideal $\implicit(f_1, \dots, f_n)$.

\item The set $S_2$ generates the ideal $\bar{J} \cap P[x_0]$.

\item  We have $S_3 = S_2$.

\item  We have $S_4 = S_3$.

\item  We have $S_1 = \{ A\dehx0 \ \mid \ A \in S_4\}$.
\end{enumerate}

\smallskip
\noindent
Claim~(1) is just the definition of $\implicit$.

To prove claim~(2) we recall that the ideal~$\bar{J}$ is homogeneous, hence
$\bar{J} \cap P[x_0]$ is too.  Thus $\bar{J} \cap P[x_0]$ can be generated by
homogeneous elements, and $S_2$ contains all homogeneous elements in
$\bar{J} \cap P[x_0]$, and so it surely generates the ideal.

We now prove claim~(3).
Let $A(x_0, x_1, \dots, x_n)$ be a homogeneous polynomial,
and let $D = \deg(A)$.  
Repeated application of Lemma~\ref{lemma:homogfractions} 
on the monomials in $A$ shows that
$$A(Q, P_1, \ldots, P_n) 
= A(q\homhD{d}, p_1\homhD{d}, \ldots, p_n\homhD{d})
= (A(q, p_1,\ldots,p_n))\homhD{dD}$$
whence $A(Q, P_1, \ldots, P_n) = 0$ if and only if $A(q, p_1,\ldots,p_n) = 0$.


Finally, we prove claim~(4). 
We have
$$
A =  x_0^{\deg(A)}A \bigl( 1,\tfrac{x_1}{x_0}, \dots, \tfrac{x_n}{x_0} \bigr)
=x_0^{\deg(A)}A\dehx0 \bigl( \tfrac{x_1}{x_0}, \dots, \tfrac{x_n}{x_0} \bigr)
$$
consequently
$A(q, p_1, \dots, p_n) = q^{\deg(A)}A\dehx0(\frac{p_1}{q}, \dots, \frac{p_n}{q})$,
hence the claimed equality follows.
Since claim (5) is clear, the proof is complete.
\end{proof}

\begin{corollary}
If \,$\implicit(\frac{p_1}{q},\dots, \frac{p_n}{q}) = \ideal{g}$ then 
$\bar{J} \cap K[x_0,x_1,\dots, x_n]= \ideal{g\hom}$.
\end{corollary}

\begin{proof}
The proof can be done exactly 
as the proof of Corollary~\ref{cor:princpoly}.a.
\end{proof}

\begin{remark}
  We may relax the restriction in the theorem that each $p_i$ be non-zero; it is
  there just to allow an easy definition of~$d$, the upper bound for the degrees.
\end{remark}

\begin{remark}\label{rem:ratgeneralhomog}
We could be tempted to use the 
general method highlighted in Remark~\ref{rem:generalfrac}:
namely, we homogenize the
generators of~$J$ given there to get the eliminating ideal~$\bar{J}^\dag$, and then 
imitate Proposition~\ref{prop:homogpoly}.
However, even if the ideal $(\bar{J}^\dag \cap P[h])\deh$ is principal, the ideal 
$\bar{J}^\dag \cap P[h]$ need not be principal, as the following example shows.
The main drawback is that $\bar{J}^\dag$ need not be saturated with respect to~$h$.
\end{remark}

\begin{example}[Ex~\ref{uqminusone} continued]\label{exnotsaturated}
We return to Example~\ref{uqminusone}, but this time homogenize the 
generators of $J=I +\ideal{ut-1}$ to produce the following 
ideal $\bar{J}^\dag = \ideal{ tx-hs,\, ty-hs,\, z-s,\, ut-h^2}$.
However, elimination yields $\bar{J}^\dag \cap K[x,y,z,h] =
 \ideal{xzh -yzh,\  xh^2 -yh^2}$
which is not principal.
Even if we homogenize the generators and 
bring them all to the same degree, we get the ideal 
$\bar{J}^\ddag =  \ideal{ tx-hs, \  ty-hs, \ h(z-s), \ ut-h^2}$,
and again elimination produces $\bar{J}^\ddag \cap K[x,y,z,h] =
 \ideal{xzh -yzh,\  xh^2 -yh^2}$.
\end{example}

%

\begin{remark}\label{weightedhomgnotgood}
In the case of rational functions, we could also be tempted to homogenize the input
in a similar way to Theorem~\ref{th:homogfractions} but applying just $\homh$
instead of equalizing
the degrees with~$\homhD{d}$.
This does not work because claim~(4) of the proof fails,
as the following easy example shows. 
\end{remark}

\begin{example}\label{ex:t2t3t4}
In the ring $K(t_1,t_2)[x_1, x_2, x_3]$ we consider the 
eliminating ideal 
$$
I = \ideal{ x_1-\frac{t_2^2}{t_1}, \;\; x_2-\frac{t_2^3}{t_1},
\;\; x_3-\frac{t_2^4}{t_1} }
$$
The correct answer is
$\implicit(\cdots) = \ideal{x_1x_3-x_2^2}$. However, if we consider the ring $K[x_0,x_1, x_2,x_3, t_1,t_2]$
graded by setting $\deg(x_0){=}1$, $\deg(x_1) {=} 2$, $\deg(x_2) {=} 3$, $\deg(x_3) {=} 4$, and
$\deg(t_1)=\deg(t_2)=1$, then the ideal $\bar{J} =\ideal{ x_0-t_1,\; x_1-t_2^2,\; x_2-t_2^3,\;
x_3-t_2^4 }$ is homogeneous, but the polynomial of minimal degree in 
$K[x_0,x_1,x_2,x_3]$ is $x_3-x_1^2$ whose degree is $4$ while the actual solution is
the polynomial $x_1x_3-x_2^2$ whose degree is $6$.
\end{example}

\goodbreak
We now turn Theorem~\ref{th:homogfractions} into an explicit algorithm:
\begin{algorithm} 
\textbf{(RatPar: Rational Parametrization)}

\begin{description}[topsep=2pt,parsep=-1pt]
\item[Input] $f_1=\frac{p_1}{q},\dots,f_n=\frac{p_n}{q} \in K(t_1,\dots,t_s)$ where~$q$ is a common denominator.

\item[RatPar-1]
Let $d =\max\{\deg(q), \deg(p_1), \dots, \deg(p_n)\}$, taking $\deg(0) = 0$ if necessary.

\item[RatPar-2]
Create the polynomial ring
$R = K[t_1,\dots,t_s,\; h,\; x_0, \; x_1,\dots,x_n]$,\\
graded by 
${\deg(t_i) = \deg(h) = 1}$ for $i=1, \dots, s$,
and
$\deg(x_i)=d$ for $i = 0,\dots, n$.


\item[RatPar-3] Let $Q=q\homhD{d} \in R$, and let $P_i =
  p_i\homhD{d} \in R$ for $i =1, \dots, n$.

\item[RatPar-4]
Compute $\ideal{G_1,\dots,G_m} = \implicit(Q, P_1, \dots, P_n)$.

\item[RatPar-5]
Compute $g_i = G_i\dehx0$ for all $i = 1, \dots, m$

\item[Output] $(g_1,\dots,g_m) \subseteq K[x_1,\dots,x_n]$~---~satisfying $\ideal{ g_1,\dots,g_m } = \implicit(f_1,\dots,f_n)$.
\end{description}
\end{algorithm}

\begin{remark}
In step \textbf{RatPar-4} we may use any algorithm to compute the implicitization from the (homogeneous)
polynomial parametrization, \eg~Algorithms~\ref{alg:ElimTH} 
 or~\ref{alg:direct} 
\end{remark}


\section{Modular Approach for Rational Coefficients}
\label{sec:RationalReconstruction}

It is well known that computations with coefficients in $\QQ$ can be
very costly in terms of both time and space.
When possible, it is generally a good idea to perform the computation
modulo one or more primes, and then ``lift'' the coefficients of
these modular results to coefficients in $\QQ$.  There are two
general classes of method: Hensel Lifting and Chinese Remaindering.
We shall use Chinese Remaindering.

The modular approach has been successfully used in numerous contexts:
polynomial factorization~\cite{Z}, determinant of integer
matrices~\cite{ABM}, ideals of points~\cite{ABKR}, and so on.  In any
specific application there are two important aspects which must be
addressed before a modular approach can be adopted:
\begin{itemize}[parsep=-0pt]
\item
  knowing how many different primes to consider to guarantee the result
  (\ie~find a realistic bound for the size of coefficients in the answer);
\item
handling \textit{bad primes}: \ie~those whose related computation follows a
different route, yielding an answer with the wrong ``shape''
(\ie~which is not simply the modular reduction of the correct non-modular result).
\end{itemize}
There is no universal technique for addressing these issues.  For our
particular application there is no useful coefficient bound, and only a
partial criterion for detecting bad primes (see Remark~\ref{rem:detecting}).
We shall use \textit{fault-tolerant rational recovery} to overcome our
limited knowledge about these two aspects
(see section~\ref{sec:MultiplePrimeMethod}).

\begin{definition} \textbf{(reduction modulo $p$)}
  Given a prime number~$p$ we denote the  usual ``reduction mod $p$'' ring homomorphism
  by $\psi_p:~\ZZ \To \ZZ/{\ideal{p}}$.
  We can extend $\psi_p$ naturally to $\ZZ[t_1,\ldots,t_s]$
  by mapping the coefficients but preserving the power-products,
  and extend it further to rational functions (over $\QQ$) by localizing
  away from its kernel in $\ZZ[t_1,\ldots,t_s]$.
  


\end{definition}

Our aim is to reconstruct the monic generator of
$\implicit(f_1,\dots,f_n)$ in $\QQ[x_1,\dots,x_n]$ from
the modular implicitizations
$\implicit\bigl(\psi_p(f_1),\dots,\psi_p(f_n)\bigr)$ 
in $\ZZ/\ideal{p}[x_1,\dots,x_n]$.

\subsection{Bad Primes}


Let $f_1,\ldots,f_n \in \QQ(t_1,\ldots,t_s)$ be non-constant and such that
$\implicit(f_1,\ldots,f_n) = \ideal{g}$ is a principal ideal, for some $g
\in \QQ[x_1,\ldots,x_n]$.  Clearly the generator~$g$ is defined only up to a non-zero
scalar multiple; we resolve this ambiguity by requiring~$g$ to be monic
(with respect to some fixed term-ordering on $\QQ[x_1,\ldots,x_n]$).  We
can now define $\den(g) \in \ZZ$ to be the least common denominator of the
coefficients of~$g$.


\begin{definition}
\label{def:unsuitable}
  We say that the prime~$p$ is \textbf{unsuitable} if any of the following happens:
\begin{description}[parsep=-0pt]
\item[(a)] there is an index~$i$ such that $f_i$ is not in the domain of $\psi_p$.
\item[(b)] there is an index~$i$ such that $\psi_p(f_i) = 0$ or\\
  $\deg(\psi_p(\numer(f_i))) < \deg(\numer(f_i))$ or \\
  $\deg(\psi_p(\denom(f_i))) < \deg(\denom(f_i))$.
\end{description}
In other words~$p$ is unsuitable if it divides any denominator,
or if the degrees of numerator and denominator of some $f_i$ change modulo~$p$.
It is easy to check whether~$p$ is unsuitable.

We exclude all unsuitable primes from subsequent discussions.
\end{definition}

\begin{definition}
We say that the prime~$p$ is \textbf{bad} if it is suitable but either of the following happens:
\begin{description}[parsep=1pt]
\item[(A)] $g$ is not in the domain of $\psi_p$, that is $p$ divides a denominator in $g$.
\item[(B)] $\implicit(\psi_p(f_1),\dots,\psi_p(f_n)) \neq \ideal{ \psi_p(g) }$.
\end{description}

\smallskip
We say that a prime is \textbf{good} if it is neither unsuitable nor bad.
We say that~$p$ is \textbf{very-good} if it is good and $\Supp(g) = \Supp(\psi_p(g))$;
in other words, it does not divide the numerator of any coefficient in~$g$.
\end{definition}

\begin{example}[Bad primes]
\label{ex:badprimes}

Given $f_1=t_1^3, \; f_2=t_2^3, \;f_3=t_1+t_2 \in \QQ[t_1,t_2]$
we have
$$\implicit(\cdots)=
\ideal{-x_3^9 +3x_1x_3^6 +3x_2x_3^6 -3x_1^2x_3^3
  +21x_1x_2x_3^3 -3x_2^2x_3^3 +x_1^3 +3x_1^2x_2
  +3x_1x_2^2 +x_2^3}
$$
but modulo~$3$ we obtain
$$\implicit(\psi_3(f_1),\;\psi_3(f_2),\;\psi_3(f_3))= 
\ideal{x_3^3 -x_1 -x_2}
\quad\subseteq\quad \ZZ/\ideal{3}[x_1,x_2,x_3]$$
So the prime~$3$ is bad because, even though the modular implicitization is
principal, it is not equivalent modulo~$3$ to the correct result.

\medskip
Indeed, even when
$\implicit(f_1,\dots,f_n)$ is principal 
in $\QQ[x_1,\dots,x_n]$
we cannot be sure that
$\implicit(\psi_p(f_1),\dots,\psi_p(f_n))$ is principal too.
For instance, given the parametrization $f_1=t_1+t_2, \;f_2=t_1-t_2$
and $f_3=t_1{-}t_2\in\QQ[t_1, t_2]$ we have
$$\implicit(f_1, f_2, f_3) = \ideal{x_2-x_3}$$
whereas modulo~$2$ we find that
$$
\implicit(\psi_2(f_1),\; \psi_2(f_2),\; \psi_2(f_3)) = 
\ideal{x_1-x_3,\; x_2-x_3}\quad\subseteq\quad \ZZ/\ideal{2}[x_1,x_2,x_3]$$
From
Remarks~\ref{rem:ElimTH-bad-input} and~\ref{rem:direct-bad-input},
we see that
in cases such as this, where the modular inputs do not satisfy the
assumption that $\implicit(\cdots)$ be principal, our 
Algorithms~\ref{alg:ElimTH}~(ElimTH)
and~\ref{alg:direct}~(Direct)
 for computing $\implicit(\psi_p(f_1), \ldots, \psi_p(f_n))$
will simply return the first polynomial in the ideal that they find.
\end{example}

\begin{remark}[Finitely many unsuitable primes]
\label{FinitelyManyUnsuitablePrimes}
Condition~(a) is satisfied if and only if $p$ divides the least common
denominator for all the $f_i$; clearly there are only finitely many
such primes.  Condition~(b) is satisfied if and only if~$p$ divides
the least common multiple of the integer contents of the leading forms
of the numerator and denominator of each $f_i$; again, clearly there are
only finitely many such primes.
\end{remark}

\begin{remark}[Finitely many bad primes]
\label{FinitelyManyBadPrimes}
  Clearly condition~(A) covers only finitely many primes.  For condition~(B)
  we consider what happens when Algorithm~\ref{alg:direct}~(Direct) runs.  We
  have a faithful modular implicitization if and only if the check for a linear dependency in
  step \textbf{Direct-2.3} actually finds one on the same iteration
  that it would have been found while computing over $\QQ$.  This will
  happen only if there was no linear dependency in any previous
  iteration; in other words, if the matrix had been of full rank in
  the penultimate iteration; and this happens for all primes except
  those which divide the numerators of all maximal minors~---~there
  are clearly only finitely many such primes.
\end{remark}

\begin{remark}
  Only finitely many primes are good but not very-good.
  By definition a prime is good but not very-good if it divides
  the numerator of some coefficient of~$g$, or equivalently if it divides the least common
  multiple of the numerators of the coefficients of~$g$.  Clearly only finitely many primes
  do so.  In conclusion, only finitely many primes are not very-good.
\end{remark}

\begin{remark}[{Detecting bad primes}]
\label{rem:detecting}
We do not have an absolute means of detecting bad primes, but given the
implicitizations modulo two different primes we can sometimes detect
that one of them is surely bad (without being certain that the other
is good).  What we can say depends on which algorithm we used to
compute the implicitizations~---~we must use the same algorithm for
both modular computations!

If we run Algorithm~\ref{alg:ElimTH}~(ElimTH) with a bad prime~$p$ to produce
the output~$g_p$ then we know that $\deg(g_p) \le \deg(g)$.  Thus if we run
Algorithm~\ref{alg:ElimTH} 
with two different primes~$p_1$ and~$p_2$,
and if $\deg(g_{p_1}) < \deg(g_{p_2})$ then surely $p_1$ is a bad
prime.  Note that even if $\deg(g_p) = \deg(g)$, we need not have $g_p
= \psi_p(g)$ as shown by the non-principal ideal in Example~\ref{ex:badprimes} above.

If we run Algorithm~\ref{alg:direct}~(Direct) with a bad prime~$p$ to produce the
output $g_p$ then we know that $\LPP(g_p) <_{\sigma} \LPP(g)$ provided we use
the same, fixed \compatto term-ordering $\sigma$.  Thus if we
run Algorithm~\ref{alg:direct} 
 with two different primes~$p_1$ and~$p_2$,
and if $\LPP(g_{p_1}) <_\sigma \LPP(g_{p_2})$ then surely $p_1$ is a bad prime.
\end{remark}

\subsection{Single Prime Method}

Given input $f_1,\ldots,f_n$ we can pick a suitable prime~$p$, and run one
of our algorithms to get an output $g_p$.
If~$p$ is very-good then $\Supp(g_p) = \Supp(g)$.
We can then determine the coefficients of~$\monic(g)$ by solving a linear
system over $\QQ$.

Let $N = | \Supp(g_p) |$ and pick $N$ distinct $s$-tuples of random integers;
evaluating all the $f_i$ at each such $s$-tuple produces a ``random point''
on the hypersurface, \ie~a zero of~$g$.  If the $N \times N$ matrix whose $(i,j)$-entry
is the value of the $i$-th power-product (in $\Supp(g_p)$) at the $j$-th tuple is of full rank
then knowing that every point on the hypersurface is a zero of~$g$,
and knowing that the leading coefficient of $\monic(g)$ is~$1$ we can
solve the linear system to get all coefficients of $g\calc$, our
``informed guess'' for the value of $\monic(g)$.

We must now verify that $g\calc$ is correct; we do this by simply
substituting $f_1,\ldots,f_n$ into it.  If our choice of prime was
very-good then the substitution will verify that $g\calc$ is
correct.  Conversely, if our choice of prime~$p$ was not very-good
then the candidate ``informed guess'' for the support of $\monic(g)$
was wrong, and $g\calc$ will lie outside $\implicit(f_1,\ldots,f_n)$, so
the substitution will give a non-zero result; in this case we must start
again with a different prime, hoping that this time it will be very-good.

This technique is advantageous when the implicitization is especially sparse
(since then the linear system will be small).

\begin{remark}
  We can make a cheaper initial verification by picking another random point on
  the hypersurface, and verifying that that point is a zero of $g\calc$.
  Naturally, if this ``randomized check'' passes then a full verification must still be done.
\end{remark}


\subsection{Multiple Prime Method}
\label{sec:MultiplePrimeMethod}

A disadvantage of the single prime method is that if the prime chosen
is not very-good then we discover this only at the end of a
potentially expensive verification.  We can greatly reduce the risk of
a failed verification by using several different primes, and combining
the corresponding modular answers using Chinese Remaindering.  Our
strategy must handle bad primes.  Using the checks in
Remark~\ref{rem:detecting} we can detect and discard some bad primes,
however it is possible that a few bad primes pass undetected.  We
use fault-tolerant rational reconstruction to cope with
any undetected bad primes; we will find the right answer so long as the
good primes sufficiently outnumber the undetected bad ones.

Moreover, when using several primes we do not require that any of the
primes be very-good; it is enough for most of the primes to be good
and ``complementary'' (\ie~the union of the supports of the answers
from all the good primes tried must include the support of the true
answer).

The key ingredient in this approach is a \textit{fault-tolerant
  rational reconstruction} procedure (\eg~see~\cite{Abb2015}
and~\cite{BDFP2015}): this enables rational coefficients to be
reconstructed from their modular images even if some of those
images are bad.  The reconstruction procedure normally returns
either the correct rational or an indication of failure, though
there is a low probability of it producing an incorrect rational.
So for certainty, the reconstructed implicit polynomial
must be verified.

We chose the HRR algorithm from~\cite{Abb2015} because it is better
suited to our application: compared to ETL from~\cite{BDFP2015} it
requires fewer primes (and therefore fewer costly modular
implicitizations) when reconstructing ``unbalanced'' rationals,
\ie~whose numerator and denominator have differing sizes.


\begin{algorithm}
\label{alg:modular}
\textbf{(ModImplicit)}
\begin{description}[topsep=0pt,parsep=1pt]
\item[Input] $f_1,\ldots,f_n \in \QQ(t_1, \dots, t_s)$ such that $\implicit(f_1,\ldots,f_n)$ is principal.
\item[ModImplicit-1]
  Fix a term-ordering $\sigma$ on the power-product monoid of $\QQ[x_1, \dots, x_{n}]$;\\
  choose an \compatto ordering if using
  Algorithm~\ref{alg:direct}(Direct) in steps~3 and~5.2.
  \item[ModImplicit-2] Choose a suitable prime $p$~---~see Definition~\ref{def:unsuitable}.
  \item[ModImplicit-3] Compute $g_p$, the monic generator
    of $\implicit(\psi_p(f_1),\dots,\psi_p(f_n))$.
  \item[ModImplicit-4] \; Let $g\crt = {g_p}$ \; and \; $\pi = p$.
  \item[ModImplicit-5] \textit{Main Loop:}
  \begin{description}[topsep=0pt,parsep=0pt]
  \item[ModImplicit-5.1] Choose a new suitable prime $p$ so all $f_i$ lie in the domain of $\psi_p$.
  \item[ModImplicit-5.2] Compute the monic generator $g_p$
    of $\implicit(\psi_p(f_1),...,\psi_p(f_n))$.
  \item[ModImplicit-5.3] Let 
$\tilde\pi = \pi\cdot p$, and
$\tilde{g}\crt$ be the polynomial whose coefficients are obtained by Chinese
Remainder Theorem from the coefficients of $g\crt$ and $g_p$.
\item[ModImplicit-5.4] Compute the polynomial
  $g\calc \in \QQ[x_1,\dots,x_n]$ whose coefficients are
  obtained as the fault-tolerant rational reconstructions of the
  coefficients of $\tilde{g}\crt$ modulo $\tilde{\pi}$.
\item[ModImplicit-5.5] Were all coefficients  ``reliably''
  reconstructed?
    \begin{description}[parsep=-3pt]
    \item[yes] if $g\calc \neq 0$ and $g\calc(f_1,\dots,f_n)=0$  exit from loop
    \item[no] \; Let $g\crt = \tilde{g}\crt$ \; and \; $\pi = \tilde\pi$
    \end{description}
  \end{description}
\item[Output] $g\calc \in \QQ[x_1, \dots,x_n]$ which generates $\implicit(f_1,\ldots,f_n)$
\end{description}
\end{algorithm}

\begin{proof} \  
  \textit{Correctness:}
Let $g \in \QQ[x_1,\ldots,x_n]$ be the monic generator of $\implicit(f_1,\ldots,f_n)$.

From the test in step~\textbf{ModImplicit-5.5} we have that 
$g\calc(f_1,\dots,f_n)=0$, so the value returned is surely an
element of $\implicit(f_1,\ldots,f_n)$; consequently, $g\calc$
is a non-zero multiple of~$g$.

We show by contradiction that~$g\calc$ is a scalar multiple of~$g$.
Suppose not, then $g\calc = f\,g$ for some non-constant polynomial~$f$.
Let $\sigma$ denote the \compatto term-ordering used inside Algorithm~\ref{alg:direct}~(Direct);
and let $\deg^*$ denote the weighted degree used inside Algorithm~\ref{alg:ElimTH}~(ElimTH)~---~note
that condition~(b) in our definition of ``unsuitable'' makes sure that the same
weighted degree is used every time.

Let $T_1 = \LPP_\sigma(g\calc)$, then clearly $T_1 >_\sigma \LPP_\sigma(g)$.
Let $T_2$ be a term of $g\calc$ of maximal weighted degree; then $\deg^*(T_2) > \deg^*(g)$.
Note that $T_1$ and $T_2$ could be the same term.
Since step~\textbf{ModImplicit-5.4} succeeded in reconstructing $g\calc$
more than half the modular implicitizations had non-zero coefficients for
the term $T_1$, and similarly for the term $T_2$.  So at least one
modular implicitization, $g_p$, had non-zero coefficients for both $T_1$ and $T_2$,
but this $g_p$ cannot have been produced by Algorithm~\ref{alg:direct}~(Direct) because
it has $\LPP_\sigma(g_p) \ge_\sigma T_1 >_\sigma \LPP_\sigma(g)$, and it cannot have been
produced by Algorithm~\ref{alg:ElimTH}~(ElimTH)
 because $\deg^*(g_p) \ge \deg^*(T_2) > \deg^*(g)$.
Thus $g\calc$ is just a scalar multiple of $g$.






\smallskip
\textit{Termination:}
The HRR algorithm in~\cite{Abb2015} for fault-tolerant rational reconstruction
guarantees to produce the correct output when the product of the good primes
is sufficiently greater than the square of the product of the bad primes
(see Corollary~3.2 in that article).

As there are only finitely many bad primes (see Remark~\ref{FinitelyManyBadPrimes}), the product
of the good primes chosen in the \textit{Main Loop} will eventually become
arbitrarily large compared to the square of the product of all bad primes (which is an
upper bound for the square of the product of all bad primes encountered in the \textit{Main Loop}).
Thus the reconstruction in step~\textbf{ModImplicit-5.4} will eventually
produce $g\calc = g$.
\end{proof}

\begin{remark}
  We can use the comments in Remark~\ref{rem:detecting} to discard some bad primes.
  If we always use Algorithm~\ref{alg:direct}~(Direct)
 to compute $g_p$ then we may insert the following step:
\begin{description}
  \item[ModImplicit-5.2a] If $\LPP(g_p) <_\sigma \LPP(g\crt)$ then go to step~\textbf{5.1}.\\
    If $\LPP(g\crt) <_\sigma \LPP(g_p)$ then set $g\crt = g_p$ and $\pi = p$; go to step~\textbf{5.1}.
\end{description}

\noindent
  If we always use Algorithm~\ref{alg:ElimTH}~(ElimTH)
 to compute $g_p$ then we may insert the following step:
\begin{description}
  \item[ModImplicit-5.2a] If $\deg(g_p) < \deg(g\crt)$ then go to step~\textbf{5.1}.\\
    If $\deg(g\crt) < \deg(g_p)$ then set $g\crt = g_p$ and $\pi = p$; go to step~\textbf{5.1}.
\end{description}
\end{remark}


\begin{remark}
Since each $g_p$ is defined only up to a scalar multiple,
we normalize the polynomial by making it monic; this guarantees
that for every good prime~$p$, the corresponding polynomial $g_p$
is equal to $\psi_p(g)$.
\end{remark}




\section{Timings}
\label{Timings}

In this section we show the practical merits of our algorithms.
We conducted two series of experiments, which we report in the
two tables below.

The experiments were performed on a MacBook Pro~2.9GHz Intel Core~i7, using our implementation
in \cocoa\,5.
The columns headed ``ElimTH'' and ``Direct'' report the computation times
for the respective algorithms: in each case there are separate columns
for computations over a finite field (char 32003), and over the rationals (char 0).
The column headed ``Len'' says how many terms there are in the resulting polynomial.
The symbol $\infty$ in the tables means that the computation was
interrupted after~20 minutes, and~$0$ means that the computation takes
less than $0.001$ seconds.
A horizontal line in the middle of the tables separates examples with polynomial parametrizations
from examples with rational parametrization.

\subsection{Examples from the Literature}

Table~1 contains statistics related to examples taken from the literature,
which we list in Appendix~\ref{OthersExamples}.  It shows that, with the sole exception
of Example~\ref{Wang}, they are computed in almost no time.

We found only two examples which defeated us: listed in our Appendix
as Examples~\ref{Dic2} and~\ref{Dic3}~---~originally they were Examples~5.2
and~5.3 in~\cite{BD}.  We suspect they are essentially incalculable
because the implicitizations are almost certainly polynomials of high
degree (over~100) having very many terms (over 100000).

\bigskip
\begin{center}
{\sc TABLE 1}

\medskip
\renewcommand{\arraystretch}{1.1}
\begin{tabular}{|l ||c|c||c|c|| r |}
\hline 
                               & ElimTH    & ElimTH  & Direct    & Direct &          \\
 Examples              &  32003  &  0        &  32003   &  0       &  Len  \\ 
\hline 
Ex \ref{d'Andrea}   &        0     &  0.009     &        0       &  0.003 &  6  \\ 
Ex \ref{Orecchia}   &        0     &  0.007     &        0       &  0.002 &  9   \\ 
Ex \ref{Enneper}    &        0     &  0.028     &        0       &  0.026 &  57  \\
Ex \ref{Robbiano}  &  0.273  &  0.597     & 0.031     &  0.118 &  319   \\
Ex \ref{Buse1}       &       0      &  0.021     &       0        &  0.070 &  13 \\
Ex \ref{Buse2}       &       0      &  0.228     &       0        &  0.083 &  56  \\
Ex \ref{Wang}        & 1.196   &  16.278   &  0.159    &  7.707 &  715  \\
\hline
Ex \ref{Dic1}          &     0        &  0.060  &        0       &  0.032 &  41  \\
Ex \ref{Dic4}          &     0        &  0.943  &        0       &  0.934 &  161  \\
Ex \ref{Bohemian} &          0   &  0.011  &             0  &  0.004 &  7  \\
Ex \ref{Sine}          &     0        &  0.012  &        0       &  0.010 &  7  \\
\hline
\end{tabular}
\end{center}

\bigskip

\subsection{Our Own Examples}

Table~2 contains statistics related to our own examples, which we list in
Appendix~\ref{OurExamples}.  The small numbers in brackets in the columns
for characteristic 0 are the number of moduli used in
Algorithm~\ref{alg:modular}~(ModImplicit)
 for reconstructing the rational coefficients.
The time to compute the answer is essentially the product of the number of
moduli and the time for a single finite field; the rest of the time is for
verification, which can represent more than half the total time as in
Example~\ref{ex4ABR-RatFun}.

\bigskip
\goodbreak
\begin{center}
{\sc TABLE 2}

\medskip
  \renewcommand{\arraystretch}{1.1}

\begin{tabular}{| l ||r|r||r|r|| r |}
\hline 
                             & ElimTH   & \phantom{a}ElimTH\phantom{a}  & Direct     & \phantom{aa}Direct\phantom{aa} &   Len          \\
 Examples            &  32003                     &  0                           &  32003  &  0        &                                            \\ 
\hline 
Ex~\ref{ex12ABR-Poly}   & 0.1    &{\footnotesize(5)}\hfill 0.9    & 0 &{\footnotesize(5)}\hfill  0.3  &  176  \\ 
Ex~\ref{ex13ABR-Poly}   &   2.1  &{\footnotesize(3)}\hfill 6.9    & 0.1 &{\footnotesize(3)}\hfill  0.4  & 471    \\ 
Ex~\ref{ex14ABR-Poly}   & $\infty$  &$\infty$& 8.41&{\footnotesize(5)}\hfill  58.2 &   6398    \\ 
Ex~\ref{ex15ABR-Poly}   & 20.3  &{\footnotesize(5)}\hfill    55.7  & 0.9 & {\footnotesize(5)}\hfill 3.4  &1705  \\ 
Ex~\ref{ex16ABR-Poly}   & $\infty$ &  $\infty$  & 58.4 &{\footnotesize(3)}\hfill 204.1 & 4304 \\ 
Ex~\ref{ex17ABR-Poly}   & 1.4   &{\footnotesize(3)}\hfill  4.8  & 9.1 &{\footnotesize(3)}\hfill 27.9 & 1763   \\ 
Ex~\ref{ex18ABR-Poly}   & $60.8 $  &  $\infty$         & 228.0     &  \hfill  $\infty$&   9360    \\ 
Ex~\ref{ex19ABR-Poly}   & 2.2    & {\footnotesize(3)}\hfill  9.3   & 47.3 &{\footnotesize(3)}\hfill 148.9  & 5801 \\ 
Ex~\ref{ex20ABR-Poly}   & 5.0     &{\footnotesize(6)}\hfill   71.5  & $\infty$  &$\infty$   & 6701     \\ 
Ex~\ref{ex21ABR-Poly}   & 10.2   & {\footnotesize(11)}\hfill 121.0  & 36.4  &{\footnotesize(11)}\hfill 418.5  & 2356  \\ 
\hline
Ex~\ref{ex1ABR-RatFun}   &  0.1    &  {\footnotesize(4)}\hfill  1.370  &0.1  & {\footnotesize(4)}\hfill 1.2  &62      \\ 
Ex~\ref{ex2ABR-RatFun}   &  0.6   & {\footnotesize(2)}\hfill  2.8  & 1.1  & {\footnotesize(2)}\hfill  3.8    & 57     \\ 
Ex~\ref{ex3ABR-RatFun}   & 0.6    & {\footnotesize(3)}\hfill   13.4  & 2.1  & {\footnotesize(3)}\hfill  17.9   & 115   \\ 
Ex~\ref{ex4ABR-RatFun}   & 10.4  & {\footnotesize(3)}\hfill  159.1    & 64.8  & {\footnotesize(3)}\hfill  335.0 &189   \\ 
Ex~\ref{ex5ABR-RatFun}   &  63.3 &  {\footnotesize(2)}\hfill 141.7    &46.2  & {\footnotesize(2)}\hfill 101.6 & 149   \\ 
Ex~\ref{ex6ABR-RatFun}   & 116.4 & {\footnotesize(6)} \hfill  761.4        & 202.7  &{\footnotesize(6)}\hfill 1214.4& 2692  \\
\hline
\end{tabular}
\end{center}

%
%


\bigskip\goodbreak

\section{Appendix: Implicitization Examples}
\label{appendix:examples}

In this appendix we list the test examples we used.  The symbol~$K$ is used to denote 
either the field $\mathbb F_{32003}$ or the field $\mathbb Q$.
The examples are of different types: in the first subsection there are
examples taken from the literature; in the second there are our own examples.

\subsection{Examples from the Literature}
\label{OthersExamples}
Here we collect examples taken from some papers mentioned in the references.

\begin{example}\label{d'Andrea} 
({\sc\cite{CdA}, Example 3.4})\quad
In the polynomial ring $K[t_0, t_1]$ we let
$$
\begin{array}{l}
f_1= t_0^4, \quad f_2= 6t_0^2t_1^2 -4t_1^4,\quad f_3= 4t_0^3t_1 -4t_0t_1^3
\end{array}
$$
\end{example}

\begin{example}\label{Orecchia} 
({\sc \cite{OR}, Example 3.1})\quad
In the polynomial ring $K[t_1,t_2,t_3]$ we let
$$
\begin{array}{l}
f_1= t_1t_2^2 -t_2t_3^2 \quad f_2= t_1t_2t_3 + t_1t_3^2,\quad f_3= 2t_1t_3^2 - 2t_2t_3^2,
\quad f_4= t_1t_2^2
\end{array}
$$
\end{example}

\begin{example}\label{Enneper}
({\sc \cite{EKK}, Enneper's Surface}, in Table 4)\quad
In the polynomial ring $K[s, t]$ we let
$$
\begin{array}{l}
f_1= t -\frac{1}{3}t^3 +s^2t, \quad f_2= 2-\frac{1}{3}s^3 +st^2, \quad f_3= t^2-s^2
\end{array}
$$
\end{example}

\begin{example}\label{Robbiano}
({\sc \cite{R}, Example 1.22})\quad
In the polynomial ring $K[s, t]$ we let
$$
\begin{array}{l}
f_1= s^5 -st^3 -t, \quad f_2= st^2 -s,\quad f_3= s^4 -t^2
\end{array}
$$
\end{example}

\begin{example}\label{Buse1}
({\sc \cite{BC}, Example 3.3.2})\quad
In the polynomial ring $K[s,t,u]$ we let
$$
\begin{array}{l}
f_1= s^2t +2t^3 +s^2u +4stu +4t^2u +3su^2 +2tu^2 +2u^3,\\
f_2=  -s^3 -2st^2 -2s^2u -stu +su^2 -2tu^2 +2u^3,\\
f_3 = -s^3 -2s^2t -3st^2 -3s^2u -3stu +2t^2u -2su^2 -2tu^2,\\
f_4 = s^3 +s^2t +t^3 +s^2u +t^2u -su^2 -tu^2 -u^3
\end{array}
$$
\end{example}

\begin{example}\label{Buse2}
({\sc \cite{BC}, Example 3.3.4})\quad
In the polynomial ring $K[s,t]$ we let
$$
\begin{array}{l}
f_1 = s^3 -6s^2t -5st^2 -4s^2u +4stu -3t^2u,\\
f_2 = -s^3 -2s^2t -st^2 -5s^2u -3stu -6t^2u,\\
f_3 = -4s^3 -2s^2t +4st^2 -6t^3 +6s^2u -6stu -2t^2u,\\
f_4 = 2s^3 -6s^2t +3st^2 -6t^3 -3s^2u -4stu +2t^2u
\end{array}
$$
\end{example}

\begin{example}\label{Wang}
({\sc \cite{W}, Example 13 p. 913})\quad
In the polynomial ring $K[s,t]$ we let
$$
\begin{array}{l}
f_1= s^3 +3t^3 -3s^2 -6t^2 +6s +3t -1,\quad 
f_2 = 3s^3 +t^3 -6s^2 +3s +3t,\\
f_3 = -3s^3t^3 -3s^3t^2 +15s^2t^3 +6s^3t -18s^2t^2 -15st^3 +9s^2t +27st^2 -3s^2   \\
\qquad -18st -3t^2 +3s +3t
\end{array}
$$
\end{example}
\medskip

\begin{example}\label{Dic1}
({\sc \cite{BD}, Example 5.1})\quad
In the field $K(s,t)$ we let
$$
\begin{array}{l}
f_1= \frac{st^6_{\mathstrut} +2_{\mathstrut}}{2 +{s^2t^6}^{\mathstrut}},\quad 
f_2 =  \frac{st^5 -3st^3_{\mathstrut}}{2 +{s^2t^6}^{\mathstrut}},\quad
f_3 = \frac{st^4_{\mathstrut} +5s^2t^6}{2 +{s^2t^6}^{\mathstrut}}
\end{array}
$$
\end{example}

\begin{example}\label{Dic2}
({\sc \cite{BD}, Example 5.2})\quad
In the field $K(s,t)$ we let
$$
\begin{array}{l}
f_1= \frac{{s^7+s^{47}}_{\mathstrut}}{1+st+s^{37}},\quad
f_2= \frac{{s^{37}+s^{59}}_{\mathstrut}}{1+st+s^{37}},\quad
f_3= \frac{{s^{61}}_{\mathstrut}}{1+st+s^{37}}
\end{array}
$$
\end{example}

\begin{example}\label{Dic3}
({\sc \cite{BD}, Example 5.3})\quad
In the field $K(s,t)$ we let
$$
\begin{array}{l}
f_1= \frac{{-s^{36}t +1}_{\mathstrut}}{1+st},\quad
f_2= \frac{{-t\,(-s^{38} +t)}_{\mathstrut}}{1+st},\quad
f_3= \frac{{s^{37}-t}_{\mathstrut}}{1+st}
\end{array}
$$
\end{example}

\begin{example}\label{Dic4}
({\sc \cite{BD}, Example 5.4})\quad
In the field $K(s_1,s_2)$ we let
$$
\begin{array}{l}
f_1=  {\frac{{3s_1^2s_2}_{\mathstrut} - s_1^2 - 3s_1s_2 - s_1 + s_2 + s_1^2 + s_2^2 +s_1^2s_2^2}
{{3s_1^2s_2}^{\mathstrut} - 2s_1s_2^2 - s_1^2 + s_1s_2 - 3s_1 - s_2 + 4 -s_2^2}}_{\mathstrut},\quad 
f_2 = {\frac{{2s_1^2s_2^2}_{\mathstrut} - 3s_1^2s_2 - s_1^2 + s_1s_2 + 3s_1 - 3s_2 + 2 - s_2^2}
{{3s_1^2s_2}^{\mathstrut} - 2s_1s_2^2 - s_1^2 + s_1s_2 - 3s_1 - s_2 + 4 -s_2^2}}^{\mathstrut},\\
\\
f_3 = \frac{{2s_1^2s_2^2}_{\mathstrut} - 3s_1^2s_2 - 2s_1s_2^2 + s_1^2 + 5s_1s_2 - 3s_1 -3s_2 + 4 - s_2^2}
{{3s_1^2s_2 }^{\mathstrut}- 2s_1s_2^2 - s_1^2 + s_1s_2 - 3s_1 - s_2 + 4 -s_2^2}
\end{array}
$$
\end{example}

\begin{example}\label{Bohemian}
({\sc \cite{EKZ}, Example 3 and \cite{EKK} Table 4, Bohemian Dome})\quad
In the field $K(s,t)$ we let
$$
\begin{array}{l}
f_1= \frac{1 -t^2_{\mathstrut}}{1 +{t^2}^{\mathstrut}},\quad f_2 = 
\frac{1 +2t +t^2 -s^2 -s^2t^2 +2ts^2_{\mathstrut}}{(1 +{t^2}^{\mathstrut})(1 +s^2)},\quad 
f_3 = \frac{2s_{\mathstrut}}{1 +{s^2}^{\mathstrut}}
\end{array}
$$
\end{example}

\begin{example}\label{Sine}
({\sc \cite{EKK}, Table 4, Sine Surface})\quad
In the field $K(s,t)$ we let
$$
\begin{array}{l}
f_1= \frac{2t_{\mathstrut}}{1 +{t^2}^{\mathstrut}},\quad f_2 =
 \frac{2s_{\mathstrut}}{(1 +{s^2}^{\mathstrut})},\quad 
f_3 = \frac{2s_{\mathstrut}}{1 +{s^2}^{\mathstrut}} \frac{1 -t^2_{\mathstrut}}{1 +{t^2}^{\mathstrut}} 
+ \frac{2t_{\mathstrut}}{1 +{t^2}^{\mathstrut}} \frac{1 -s^2_{\mathstrut}}{1 +{s^2}^{\mathstrut}} 
\end{array}
$$
\end{example}

\bigskip

\subsection{Our Own Examples}
\label{OurExamples}

Here are several examples we used while exploring the behaviour of our algorithms.

\begin{example}\label{ex12ABR-Poly}
In the polynomial ring $K[t]$ we let
$$
\begin{array}{l}
f_1= t^{15} -3t^2 -t +1, \quad f_2= t^{23} +t^{11} +t^3 -t -2
\end{array}
$$
\end{example}

\begin{example}\label{ex13ABR-Poly}
In the polynomial ring $K[s,t]$ we let
$$
\begin{array}{l}
f_1= st^5 -st^3 -t, \quad f_2= s^3 -st -t^2 -1,\quad f_3= s^2t^2 -s
\end{array}
$$
\end{example}

\begin{example}\label{ex14ABR-Poly}
In the polynomial ring $K[s,t]$ we let
$$
\begin{array}{l}
f_1= s^7 -st^3 -t, \quad f_2= st^3 -s,\quad  f_3=  s^{13} -t^2
\end{array}
$$
\end{example}

\begin{example}\label{ex15ABR-Poly}
In the polynomial ring $K[s,t,u]$ we let
$$
\begin{array}{l}
f_1=  s^2 -st -tu , \quad f_2=  st^2 -su, \quad f_3= s^3 -t^2 +u,\quad f_4 = s +u^2
\end{array}
$$
\end{example}

\begin{example}\label{ex16ABR-Poly}
In the polynomial ring $K[s,t,u]$ we let
$$
\begin{array}{l}
f_1= s^3 -t^2 +u,\quad f_2= s^3 -t^2 +u^2 +s +u,\quad f_3= s^5 -tu,\quad f_4=  st^2 -su
\end{array}
$$
\end{example}

\begin{example}\label{ex17ABR-Poly}
In the polynomial ring $K[s,t,u,w]$ we let
$$
\begin{array}{l}
f_1 = s^2 -t^2 +w,\  f_2= s^2 -u -w,\  f_3= s^2 -tu,\  f_4=  t^2 -su,\  f_5= s^3 + t -u -w
\end{array}
$$
\end{example}

\begin{example}\label{ex18ABR-Poly}
In the polynomial ring $K[s,t,u,w]$ we let
$$
\begin{array}{l}
f_1 = s^2 -t -u +w,\  f_2= t^2 -u -w,\  f_3= s -tu,\ f_4=  u^2 -sw,\  f_5 =\! s^2 + t -u -w^2
\end{array}
$$
\end{example}

\begin{example}\label{ex19ABR-Poly}
In the polynomial ring $K[s,t,u,v, w]$ we let
$$
\begin{array}{l}
f_1= s^2 -t -u,\  f_2 = u^2 -sw,\  f_3 = \! s^2 -v,\  
f_4 =  u^2 -v -w,\  f_5 = t -u^2,\   f_6 =\!  v^2 -w
\end{array}
$$
\end{example}

\begin{example}\label{ex20ABR-Poly}
In the polynomial ring $K[s,t,u,v, w]$ we let
$$
\begin{array}{l}
f_1= s^3 -u^2 -t -3s -u +w,\quad f_2 = u^2 -sw -11,\quad f_3 =  s^2 -5u -v,\\
f_4 =  u^2 - s -v -w,\quad f_5 = u^2 +7s +t,\quad f_6 =\!  v^2 +s^2 -s -t -w
\end{array}
$$
\end{example}

\begin{example}\label{ex21ABR-Poly}
In the polynomial ring $K[s,t,u]$ we let
$$
\begin{array}{l}
f_1= s^3 -s^2 -t^2 +3s -21,\quad f_2 = t^5 -s^5 +st^4 -st +t^2 -s -t -21,\\
f_3 =  t^3 -2t^2s -5ts^2 -t^2 +5s -12u, \quad f_4 = s +t -u
\end{array}
$$
\end{example}

\medskip
\begin{example}\label{ex1ABR-RatFun}
In the field $K(t)$ we let
$$
\begin{array}{l}
f_1=\frac{ 2t^2 -t -3_{\mathstrut}}{ 1 +{t^{17}}^{\mathstrut}} \quad 
f_2=\frac{ t^4 -t +1_{\mathstrut}}{ {t^2}^{\mathstrut} -t -1}
\end{array}
$$
\end{example}

\begin{example}\label{ex2ABR-RatFun}
In the field $K(s,t)$ we let
$$
\begin{array}{l}
f_1=\frac{ s^3 -t_{\mathstrut}}{{t^2}^{\mathstrut} -s -t }, \quad 
f_2=\frac{s -t_{\mathstrut} }{{s^3}^{\mathstrut} -2 },\quad
f_3 =\frac{ s_{\mathstrut}}{ {s^2}^{\mathstrut} +t}
\end{array}
$$
\end{example}

\begin{example}\label{ex3ABR-RatFun}
In the field $K(s,t)$ we let
$$
\begin{array}{l}
f_1=\frac{s^2_{\mathstrut} -t^2 -s}{ {t^2}^{\mathstrut} -s -t}, \quad 
f_2=\frac{s -t -4_{\mathstrut} }{ {s^3}^{\mathstrut} -2t -5},\quad
f_3 =\frac{s -2_{\mathstrut}}{ {s^2}^{\mathstrut} +t}
\end{array}
$$
\end{example}

\begin{example}\label{ex4ABR-RatFun}
In the field $K(t_1,t_2,t_3)$ we let
$$
\begin{array}{l}
f_1=\frac{t_1t_3-{t_2^2}_{\mathstrut}}{t_2-{t_3}^{\mathstrut}}, \quad 
f_2=\frac{-t_2+t_3-4_{\mathstrut}}{t_1-2t_2-5^{\mathstrut}},\quad
f_3 = \frac{t_1-2_{\mathstrut}}{{t_1^2}^{\mathstrut}+t_3}, \quad 
f_4 = \frac{{t_3}_{\mathstrut}}{t_1+t_3^{\mathstrut}}
\end{array}
$$
\end{example}

\begin{example}\label{ex5ABR-RatFun}
In the field $K(s,t,u)$ we let
$$
\begin{array}{l}
f_1=\frac{s^3 -t -u_{\mathstrut}}{ {t^2}^{\mathstrut} -s -t}, \quad 
f_2=\frac{t -u_{\mathstrut} }{ {s^3}^{\mathstrut}-2},\quad
f_3 =\frac{s^2_{\mathstrut} }{{s^2}^{\mathstrut} +t}, \quad 
f_4 = \frac{u_{\mathstrut}}{{s^2}^{\mathstrut} +t }
\end{array}
$$
\end{example}

\begin{example}\label{ex6ABR-RatFun}
In the field $K(s,t,u,v)$ we let 
$$
\begin{array}{l}
f_1=\frac{ s -t -u_{\mathstrut}}{(t^2-s-t -u -v)(s+u) }, \quad 
f_2=\frac{ t^3 -v^2_{\mathstrut}}{ ({t^2}^{\mathstrut}-s-t -u -v)(s+u)},\quad
f_3 =\frac{u -v _{\mathstrut}}{ ({t^2}^{\mathstrut}-s-t -u -v)(s+u)}, \\
f_4 = \frac{s +u -v_{\mathstrut}}{(({t^2}^{\mathstrut}-s-t -u -v)(s+u)},  
\quad {f_5 = \frac{s^2 -5u -6v_{\mathstrut}}{({t^2}^{\mathstrut}-s-t -u -v)(s+u)}}^{\mathstrut}
\end{array}
$$
\end{example}

\goodbreak

\end{document}